\documentclass{article}
\usepackage[utf8]{inputenc}

\usepackage{fancyhdr}
\usepackage{bm,amsmath,bbm,amsfonts,nicefrac,latexsym,amsmath,amsfonts,amsbsy,amscd,amsxtra,amsgen,amsopn,bbm,amsthm,amssymb, enumerate, appendix, listings,subcaption}
\usepackage{pifont}
\usepackage{graphicx}
\usepackage[nottoc,numbib]{tocbibind}

\usepackage{biblatex}
\makeatletter
\renewenvironment{proof}[1][\proofname]{%
   \par\pushQED{\qed}\normalfont%
   \topsep6\p@\@plus6\p@\relax
   \trivlist\item[\hskip\labelsep\bfseries#1\@addpunct{.}]%
   \ignorespaces
}{%
   \popQED\endtrivlist\@endpefalse
}
\makeatother

\newtheorem{lemma}{Lemma}
\newtheorem{theorem}{Theorem}[section]

\newtheorem{ex}{Example}

\title{A new formula for rotation number}
\author{Damián Wesenberg}
\date{October 3, 2020}

\begin{document}

\maketitle

\begin{abstract}
We give a new formula for the rotation number (or Whitney index) of a smooth closed plane curve. This formula is obtained from the winding numbers associated with the regions and the crossing points of the curve. One difference with the classic Whitney formula is that ours does not need a base point.
\end{abstract}

\section{Introduction}
    Informally, the rotation number $w(\gamma)$ of a regular closed planar curve $\gamma$ is just the number of complete turns the tangent vector to the curve makes as one passes once around the curve; and the winding number $wind(\gamma, p)$ of a closed curve $\gamma$ with respect to a point $p$ is the number of times the curve winds around the point (see Figure \ref{Intro}).
    In \cite{whitney1937regular} Whitney showed that the rotation number is invariant under regular homotopy. Moreover, again by Whitney \cite{whitney1937regular}, there is a simple formula for the rotation number of normal planar curves in terms of the number of positive crossings and negative crossings with respect to a base point.
    
\begin{figure}[ht]
\centering
\includegraphics[scale=0.8]{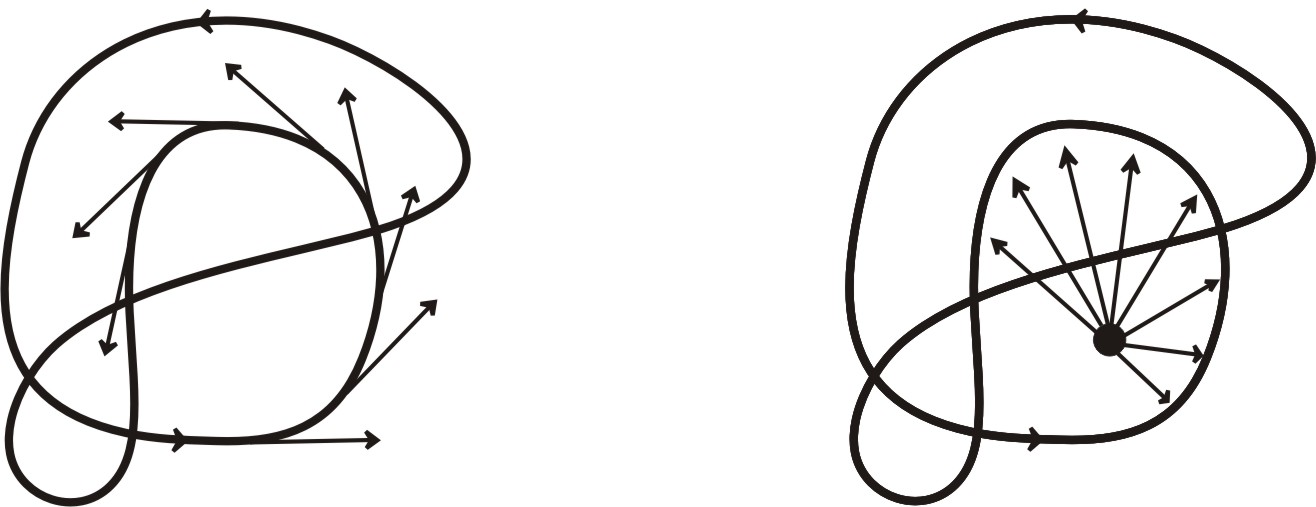}
\caption{$w(\gamma)=1$ and $wind(\gamma,p)=1$.}
\label{Intro}
\end{figure}

In the sections 2, 3 and 4 of this article we give the basic ideas about rotation number and winding numbers, and in section 5 we give our new formula for calculating the rotation number of a closed curve as a function of the winding numbers (Theorem \ref{mio}).

\section{Winding numbers.}

We will say that a curve $\gamma$ is generic (or normal) if $\gamma$ has a finite number of crossing points and they are transverse double points. If $\gamma:[0,1]\longrightarrow{\mathbb{C}}$ is a closed generic curve, and $p$ is not in the image of $\gamma $; then we may find differentiable functions, $r(t)>0$ and $\theta(t)\in\mathbb{R}$ such that

$$\gamma(t)=p+r(t)e^{i\theta(t)}$$

The function $\theta(t)$ is unique up to an additive constant, $2\pi k$; for some $k\in\mathbb{Z}$ and therefore $\Delta \theta :=\theta (1)-\theta (0)$ is well defined independent of the choice of $\theta $. Moreover, since $\gamma$ is a closed curve so that $\gamma (1)=\gamma (0)$; we have that $\theta(1)-\theta(0)=2\pi m$ for some $m\in\mathbb{Z}$.

So we can give the following definition:%

$$wind(\gamma,p):=\frac{\theta(1)-\theta(0)}{2\pi}$$

We can easily see that the function $wind(\gamma,p)$ is constant as a function of $p$ in each connected component of $\mathbb{C}-Im(\gamma )$. If $r$ is a region of the curve $\gamma $, we define $%
wind(r):=wind(\gamma ,p)$ where $p$ is any point in the region $r$.
Intuitively, $wind(r)$ is the number of times that $\gamma $ winds
counterclockwise around any point in region $r$. The interger numbers given by the following lemma (see \cite{alexander1928topological}) are the numbers $wind(r)$.

\begin{lemma}
Given a closed oriented normal curve $\gamma $, one can associate
integers to each of the regions such that at each segment of $\gamma $ the number to the left of $\gamma $ is $1$ greater than the number to the right, and the outside region is numbered zero. Moreover, such a numbering is unique.
\end{lemma}

\begin{ex}

In Figure \ref{Winding} we can see the winding numbers associated with the regions of a curve.

\begin{figure}[ht]
\centering
\includegraphics[scale=0.15]{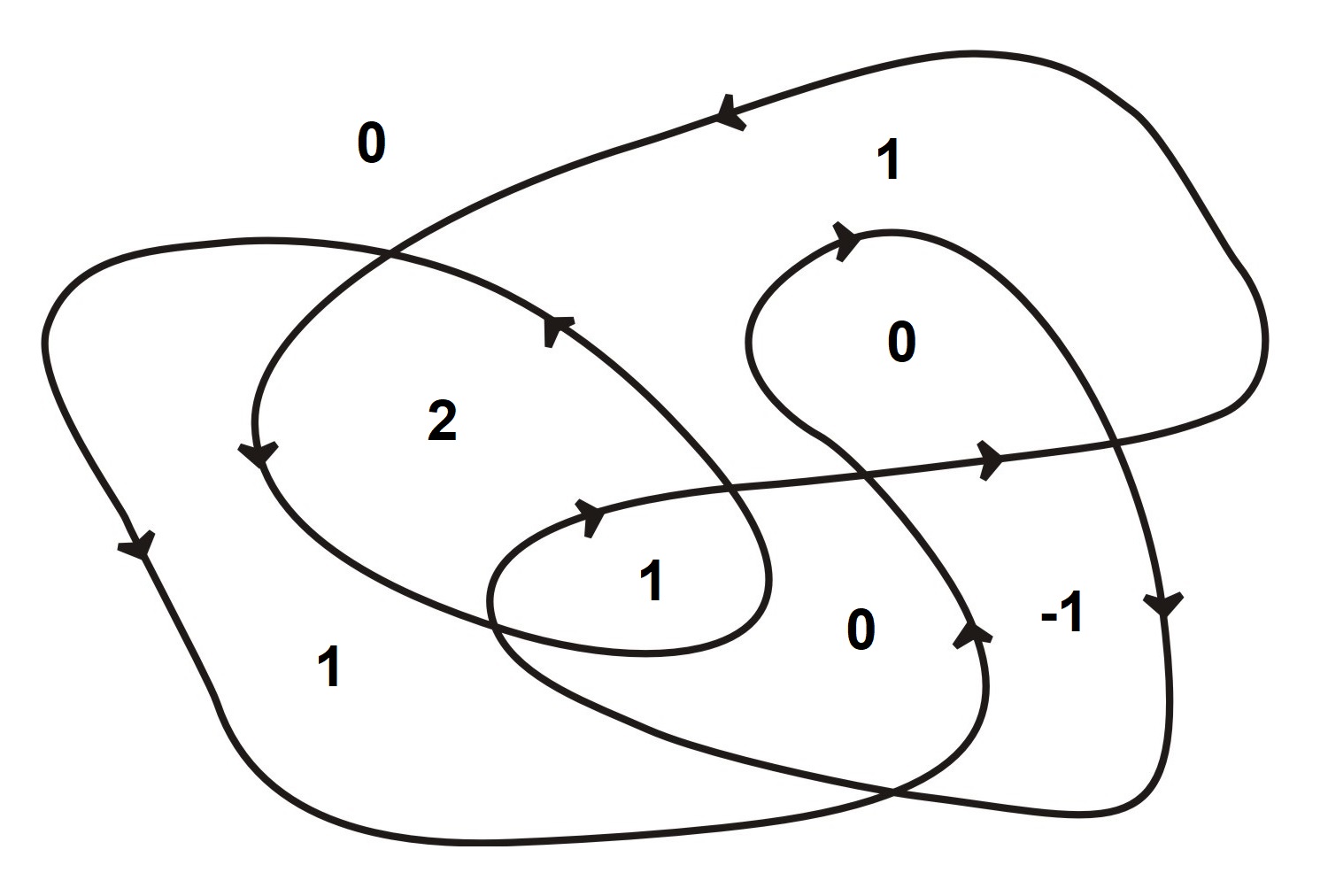}
\caption{The winding numbers of a closed curve.}
\label{Winding}
\end{figure}

\end{ex}

Another way to define the winding numbers is as follows: let $%
\rho $ be any ray from $p$ to infinity that intersects $\gamma $
transversely. The winding number of $\gamma $ around $p$ is the number of times $\gamma $ crosses $\rho $ from right to left, minus the number of times $\gamma $ crosses $\rho $ from left to right. The winding number does not depend on the particular choice of ray $\rho $ (see Figure \ref{ExampleRay}).

\begin{figure}[ht]
\centering
\includegraphics[scale=0.5]{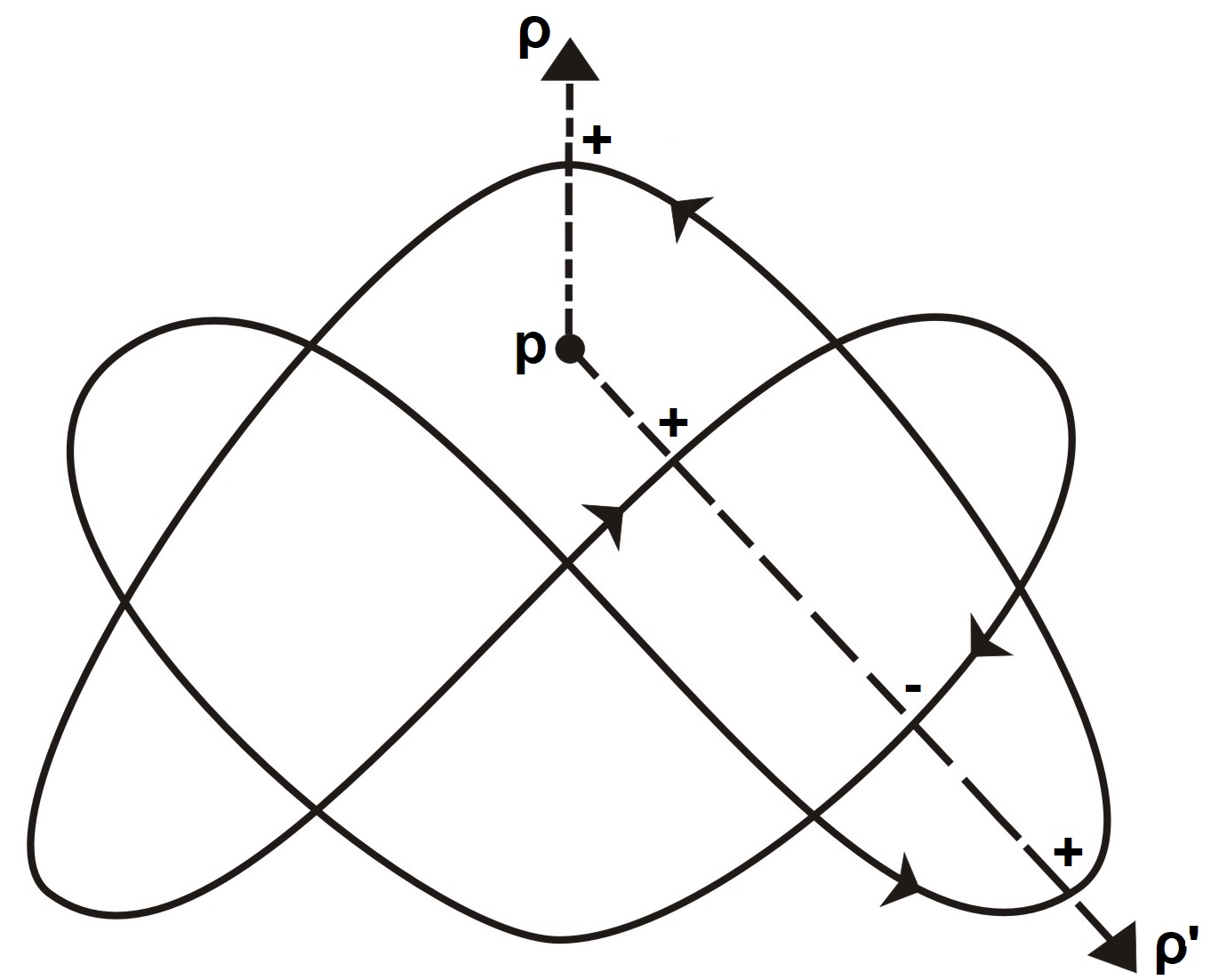}
\caption{$wind(\gamma,p)=1$.}
\label{ExampleRay}
\end{figure}

\section{The rotation number}

 We say that a closed curve $\gamma $ is regular if its derivative\ $\gamma ^{\prime }$\ exists, it is continuous, $\gamma ^{\prime }(t)\neq
0$ for all $t\in \lbrack 0,1]$ and $\gamma ^{\prime }(0)=\gamma ^{\prime
}(1) $. The rotation number $w(\gamma)$ of a regular closed plane curve $\gamma $ is the number of complete rotations that a tangent vector to the curve makes as it goes around the curve. Now we are going to give a formal definition.

Let $\gamma :[0,1]\longrightarrow%TCIMACRO{\U{211d} }%BeginExpansion
\mathbb{R}%EndExpansion
^{2}$ be a regular closed curve. We define the function

$$\overline{\gamma }:[0,1]\longrightarrow{S^{1}}$$
$$\overline{\gamma }(t)=\frac{\gamma^{\prime }(t)}{\left\vert\gamma^{\prime }(t)\right\vert }$$

Let us observe that $\overline{\gamma }(0)=\overline{\gamma }(1)$. The rotation number $w(\gamma)$ of $\gamma $ is defined as the degree of the map $\overline{%
\gamma }$, this is
\[
w(\gamma ):=degree(\overline{\gamma })
\]

\vspace{1pt}Let $x\in Im(\gamma )$ be a point (we will call it base point) that is not a crossing point. For a crossing point $c$, this determines an ordering of two outgoing branches of $\gamma $ in $c$ (the branch that passes through the crossing first and the one that passes second). If when we cross the crossing $c$ through the first branch the second branch passes from left to right, we will say that $c$ is a positive crossing and we will write $\varepsilon _{c}(x)=1$. Otherwise we will say that $c$ is negative and write $\varepsilon _{c}(x)=-1$. Also, we define $ind_{\gamma }(x)$ as the average between the values of the two winding number of the regions that are on the sides of $x$. Let us observe that $ind_{\gamma }(x)\in \frac{1}{2}%
%TCIMACRO{\U{2124} }%
%BeginExpansion
\mathbb{Z}
%EndExpansion
$.

\begin{theorem}
 
(Whitney \cite{whitney1937regular}) Let $x\in Im(\gamma )$ be a point that is not a crossing point. Then the rotation number $w(\gamma)$ satisfies  
$$w(\gamma)=\sum_{c}^{}\varepsilon_c(x)+2ind_{\gamma}(x).$$
 
\end{theorem}

\begin{ex}

On the following curve we choose a base point $x$ and calculate the signs of the crossing points. For this, we start from $x$ in the direction of the orientation of the curve until we reach the first crossing point and we see that the second branch crosses us from right to left, therefore the crossing point is negative. We continue traveling the curve and when we arrive (for the first time) at the second crossing point we see that the second branch crosses us from left to right, therefore the crossing point is positive. Also, the regions on the sides of $x$ have winding numbers equal to 0 and 1, therefore $ind_{\gamma}(x)={1}/{2}$. Then $w(\gamma)=\sum_{c}^{}\varepsilon_c(x)+2ind_{\gamma}(x)=1$.

\begin{figure}[ht]
\centering
\includegraphics[scale=0.5]{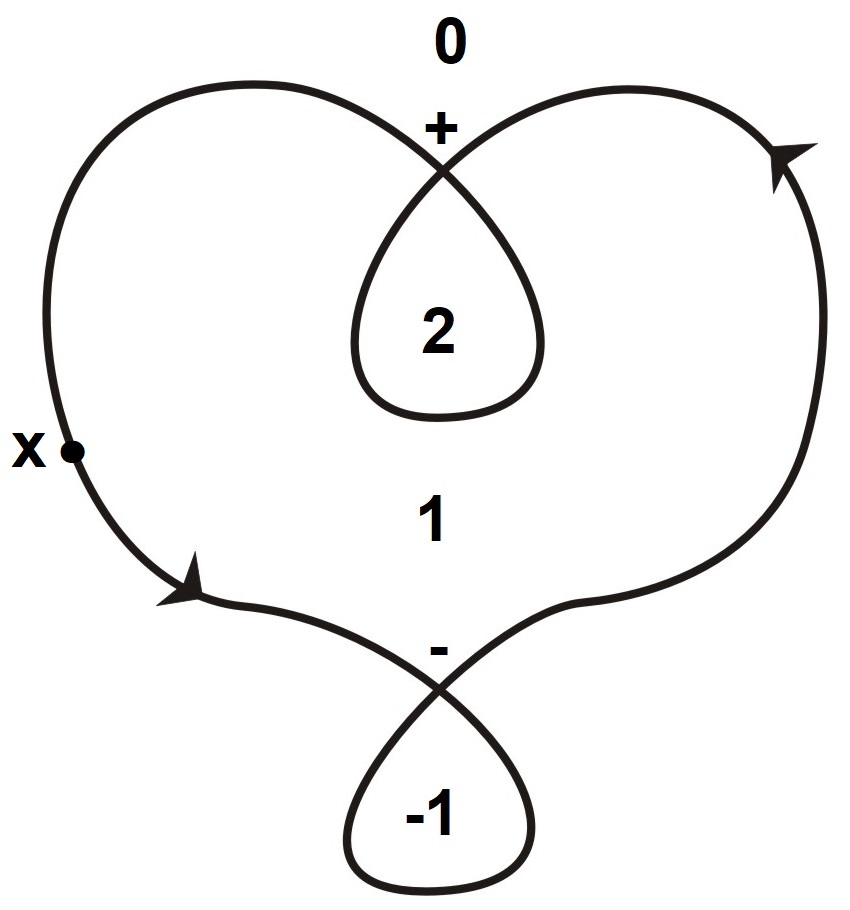}
\end{figure}

\end{ex}

\section{Regular homotopy}

A regular homotopy is a function $h:[0,1]^{2}\longrightarrow 
%TCIMACRO{\U{211d} }%
%BeginExpansion
\mathbb{R}
%EndExpansion
^{2}$ such that for all $s$, the function $t\longmapsto h(s,t)$ is a regular closed curve, and the partial derivative $\partial h/\partial t$ is a free homotopy between loops in $%
%TCIMACRO{\U{211d} }%
%BeginExpansion
\mathbb{R}^{2}
%EndExpansion
-\{0\}$. Two regular closed curves $\gamma _{1}$ and $\gamma _{2}$ are regularly homotopic if there is a regular homotopy $h$ such that $h(0,$%
\textperiodcentered $)=\gamma _{1}$ and $h(1,$\textperiodcentered $%
)=\gamma _{2}$.

\begin{theorem}
 
 (Whitney - Graustein \cite{whitney1937regular}) Two regular closed curves in $%
%TCIMACRO{\U{211d} }%
%BeginExpansion
\mathbb{R}
%EndExpansion
^{2}$ are regularly homotopic if and only if their rotation numbers are equal.
 
\end{theorem}

Given two closed generic curves $\gamma _{1}$ and $\gamma _{2}$, it is always possible to go from $\gamma _{1}$ to $\gamma _{2}$ through a finite sequence of elementary moves $M_{1}$, $M_{2}$ and $M_{3}$ (see Figure \ref{MovesM}). This elementary moves are \textquotedblleft shadows\textquotedblright\ of the classical Reidemeister moves used to manipulate knot and link diagrams in \cite{reidemeister1927elementare}. Furthermore, it is known that $\gamma _{1}$ and $\gamma _{2}$ will be regularly homotopic if and only if there exists a finite sequence of elementary moves $M_{2}$ and $M_{3}$ that connect $\gamma _{1}$ with $\gamma _{2}$.

\begin{figure}[ht]
\centering
\includegraphics[scale=0.4]{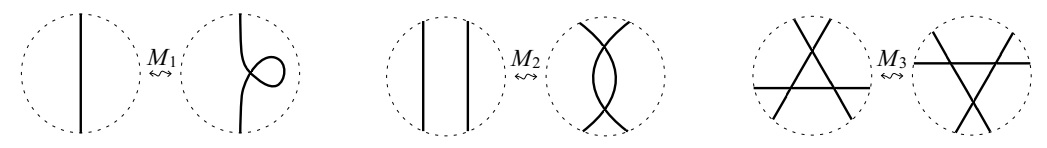}
\caption{The elementary moves $M_1$, $M_2$ and $M_3$.}
\label{MovesM}
\end{figure}

Similar things can be said for oriented closed generic curves. Given two oriented closed generic curves $\gamma _{1}$ and $\gamma _{2}$, it is always possible to go from $\gamma _{1}$ to $\gamma _{2}$ through a finite sequence of oriented elementary moves $M_{1a}$%
, $M_{1b}$, $M_{2b}$ and $M_{3a}$ (see Figure \ref{MovesOriented}). The fact that these four oriented moves is sufficient can be seen in the article \cite{polyak2009minimal} of Polyak. In this article Polyak shows that oriented Reidemeister moves $\Omega_{1a}$, $\Omega_{1b}$, $\Omega_{2a}$ and $\Omega_{3a}$ (moves of knot diagrams similar to $M_{1a}$, $M_{1b}$, $M_{2a}$ and $M_{3a}$ but indicating at each crossing the branch that passes below and the one that passes above) are enough to generate all oriented Reidemeister moves.

\begin{figure}[ht]
\centering
\includegraphics[scale=0.5]{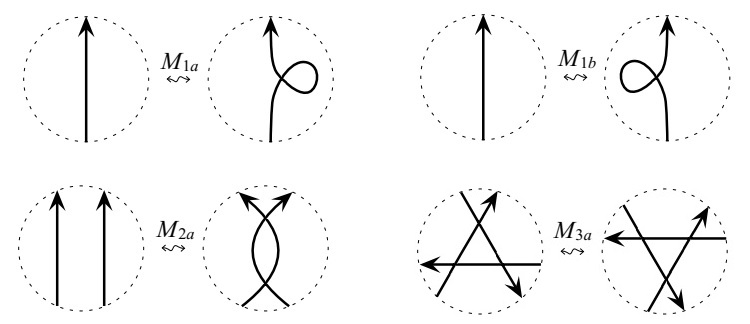}
\caption{The oriented elementary moves $M_{1a}$, $M_{1b}$, $M_{2a}$, and $M_{3a}$.}
\label{MovesOriented}
\end{figure}

\section{A new formula for the rotation number}

Now we are going to associate an integer to each crossing point of the curve. It is easy to verify that at any crossing point $c$ there are always two opposite corners with the same winding number $A$ and two opposite corners with winding number $A-1$ and $A+1$ (see Figure \ref{Crossing}). Then we associate the interger $A$ as the associated number of the crossing point $c$ and we write $wind(c)=A$.

\begin{figure}[ht]
\centering
\includegraphics[scale=0.5]{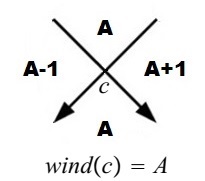}
\caption{The winding number of a crossing point.}
\label{Crossing}
\end{figure}

Combinatorially, a generic closed curve can be thought of as a 4-regular graph. If the number of vertices of this graph (the crossing points) is $n$, then the number of edges is $2n$. Also, follows from Euler's formula in the sphere ($\#faces-\#edges+\#vertices=2$) that the number of faces (regions) is $n+2$.

\begin{theorem}\label{mio}
 Let $\gamma $ be a closed oriented normal curve in the plane. Let $%
r_{1},r_{2},...,r_{n+2}$ be the regions of $\gamma $ and let $%
c_{1},c_{2},...,c_{n}$ be the crossing points of $\gamma $. Then the rotation number $w(\gamma)$ satisfies
$$w(\gamma)=\sum_{i=1}^{n+2}wind(r_i)-\sum_{i=1}^{n}wind(c_i).$$
\end{theorem}

\begin{proof}

Given a closed oriented normal curve $\gamma $ we define
$$d(\gamma):=\sum_{i=1}^{n+2}wind(r_i)-\sum_{i=1}^{n}wind(c_i).$$

We want to prove that $w(\gamma )=d(\gamma )$. Let us observe that for a Jordan curve (that is, a simple closed curve) the equality $w(\gamma )=d(\gamma )$ holds (see Figure \ref{Jordan}).

\begin{figure}[ht]
\centering
\includegraphics[scale=1]{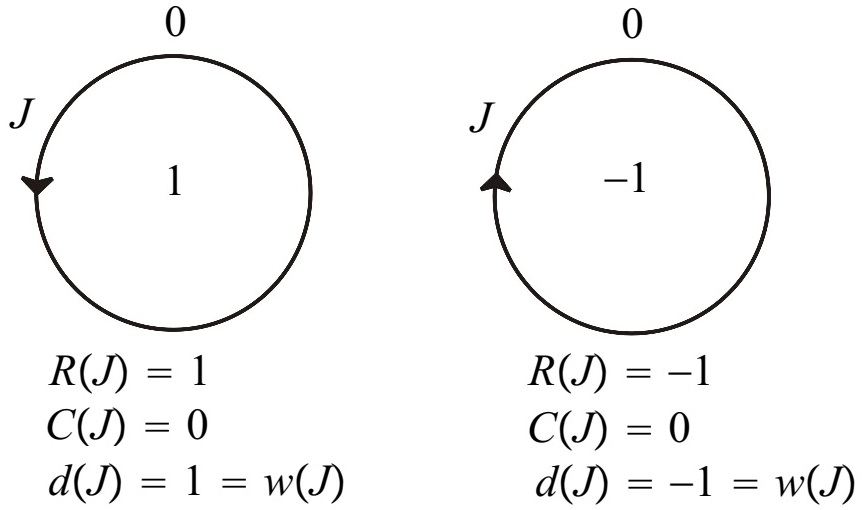}
\caption{The rotation number coincides with $d(\gamma )$ on the Jordan curve.}
\label{Jordan}
\end{figure}

We will use the following notation:

$$R(\gamma):=\sum_{i=1}^{n+2}wind(r_i)\:\:\:\:\:\:and\:\:\:\:\:\:C(\gamma):=\sum_{i=1}^{n}wind(c_i)$$

So that $d(\gamma)=R(\gamma)-C(\gamma)$. Furthermore, we know that given two curves $\gamma _{1}$ and $\gamma _{2}$ it is always possible to transform one curve into the other by performing a finite sequence of moves $M_{1a}$, $M_{1b}$, $M_{2a}$ and $M_{3a}$. Now we will see that the value of $w(\gamma )-d(\gamma )$ remains constant when we make these moves.

Let $D$ be a changing disc within which we are going to make a elementary move. We define $R_{ex}(\gamma)$ as the sum of the winding numbers of the regions that have empty intersection with $D$, and define $R_{in}(\gamma)$ as the sum of the winding numbers of the regions that have nonempty intersection with $D$. Trivially we have that $R(\gamma)=R_{ex}(\gamma)+R_{in}(\gamma)$. Similarly for the crossing points we define $C_{ex}(\gamma)$ as the sum of the winding numbers of the crossing points that do not belong to the changing disc $D$ and we define $C_{in}(\gamma)$ as the sum of the winding numbers of the crossing points that belong to the changing disc $D$. Trivially we have that $C(\gamma)=C_{ex}(\gamma)+C_{in}(\gamma)$.

If $r$ is a region that is not entirely contained in $D$, then $wind(r)$ is preserved after the elementary move is made. To understand this think of the definition of $wind(r)$ using a ray emerging from any point of $r$ (choose a point outside disk $D$) and any direction (choose a direction so that the ray does not pass through disk $D$) that cuts across $\gamma$ (see Figure \ref{ExampleRay}). The above statement implies that both $R_{ex}(\gamma)$ and $C_{ex}(\gamma)$ are preserved after the move, and therefore $d_{ex}(\gamma):=R_{ex}(\gamma) - C_{ex}(\gamma)$ is preserved. Now let's see what happens to $d_{in}(\gamma):=R_{in}(\gamma) - C_{in}(\gamma)$ when we make a elementary move. For that we will analyze each elementary move separately:

In the elementary move $M_{1a}$ (viewed from left to right) the value of $d_{in}(\gamma)$ decreases by 1 (see Figure \ref{Move1}). Then $d(\gamma)=d_{ex}(\gamma)+d_{in}(\gamma)$ decreases by 1. Furthermore, since $w(\gamma)$ also decreases by 1 we have that $w(\gamma )-d(\gamma )$ remains constant. In the elementary move $M_{1b}$ (viewed from left to right) the value of $d_{in}(\gamma)$ increase by 1 (see Figure \ref{Move1}). Then $d(\gamma)=d_{ex}(\gamma)+d_{in}(\gamma)$ increase by 1. Furthermore, since $w(\gamma)$ also increases by 1 we have that $w(\gamma )-d(\gamma )$ remains constant.

\begin{figure}[ht]
\centering
\includegraphics[scale=0.6]{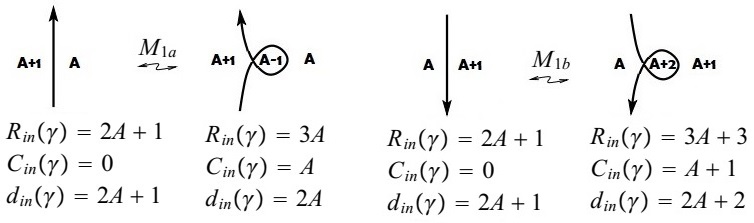}
\caption{The moves $M_{1a}$ and $M_{1b}$ do not change the value of $w(\gamma)-d(\gamma)$.}
\label{Move1}
\end{figure}

In the elementary move $M_{2a}$ el value of $d_{in}(\gamma)$ remain constant (see Figure \ref{Move2and3}). Then $d(\gamma)=d_{ex}(\gamma)+d_{in}(\gamma)$ remain constant. Furthermore, since $w(\gamma)$ also remain constant we have that $w(\gamma )-d(\gamma )$ remains constant.
In the elementary move $M_{3a}$ el value of $d_{in}(\gamma)$ remain constant (see Figure \ref{Move2and3}). Then $d(\gamma)=d_{ex}(\gamma)+d_{in}(\gamma)$ remain constant. Furthermore, since $w(\gamma)$ also remain constant we have that $w(\gamma )-d(\gamma )$ remains constant.

\begin{figure}[ht]
\centering
\includegraphics[scale=0.6]{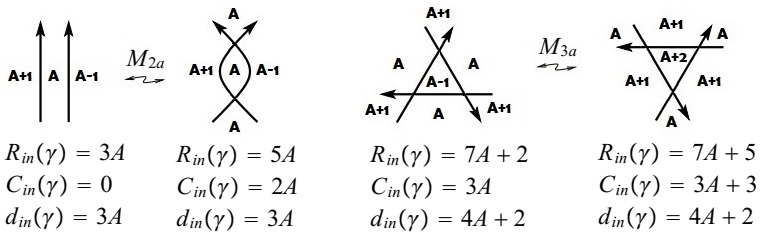}
\caption{The moves $M_{2a}$ and $M_{3a}$ do not change the value of $d_{in}(\gamma)$.}
\label{Move2and3}
\end{figure}

In Figures \ref{Move1} and \ref{Move2and3} we have assumed that all the sub-regions involved (pieces of regions within the changing disc $D$) correspond to different regions. But we must consider the cases in which there are sub-regions that are part of the same region. In the elementary moves $M_{1a}$, $M_{1b}$ and on the left side of the elementary move $M_{2a}$ this is impossible as there are not two sub-regions with the same winding number. On the right side of the elementary move $M_{2a}$ this can occurs, but if this happens it would force our curve to split in two. In the elementary move $M_{3a}$ the sub-regions 1, 3 and 5 (see Figure \ref{Sub-regions}) could be part of the same region (the same is the case with sub-regions 2, 4 and 6). Now suppose that in the elementary move $M_{3a}$ on the left side we have that sub-regions 1 and 3 are part of the same region.

\begin{figure}[ht]
\centering
\includegraphics[scale=0.6]{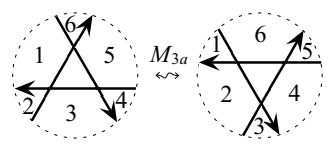}
\caption{Numbered sub-regions for the move $M_{3a}$.}
\label{Sub-regions}
\end{figure}

So on the right side, the same should happen with sub-regions 1 and 3. So to calculate the value of $d_{in}(\gamma)$ we only have to subtract $wind$(sub-region 1) from both sides of the move with respect to the calculation done in Figure \ref{Move2and3}. Therefore the value of $d_{in}(\gamma)$ remains constant in the move (see Figure \ref{SameRegion}). The same occurs if we have three sub-regions in the same region.

\begin{figure}[ht]
\centering
\includegraphics[scale=0.6]{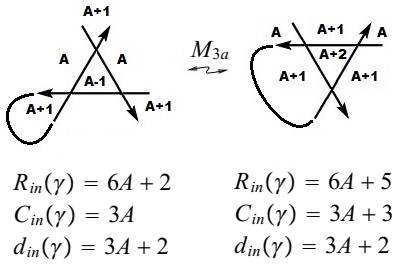}
\caption{Two pieces of regions belong to the same region.}
\label{SameRegion}
\end{figure}

We have proven that when we make any of the elementary moves $M_{1a}$, $M_{1b}$, $M_{2a}$ and $M_{3a}$ the value of $W(\gamma)-d(\gamma)$ remains constant.

Now we transform our curve $\gamma $ into a Jordan curve $J$ by a finite sequence of moves $M_{1a}$, $M_{1b}$, $M_{2a}$ and $M_{3a}$. Since in the Jordan curve holds $w(J)-d(J)=0$, then also in our curve it is worth that $%
w(\gamma )-d(\gamma )=0$. Then $w(\gamma )=d(\gamma )$.

\end{proof}

\begin{ex}
On the following curve we choose a base point $x$ and calculate the signs of the crossing points. Five are positive and three are negative (relative to base point $x$). Moreover, we have that $ind_{\gamma}(x)=1/2$, so by Whitney's theorem we have that $w(\gamma)=\sum_{c}^{}\varepsilon_c(x)+2ind_{\gamma}(x)=3$.
On the other hand, the sum of the winding numbers of the regions is equal to 11 and the sum of the winding numbers of the crossing points is equal to 8 (all are equal to 1). Thus, $d(\gamma)=3$. 

\begin{figure}[ht]
\centering
\includegraphics[scale=0.8]{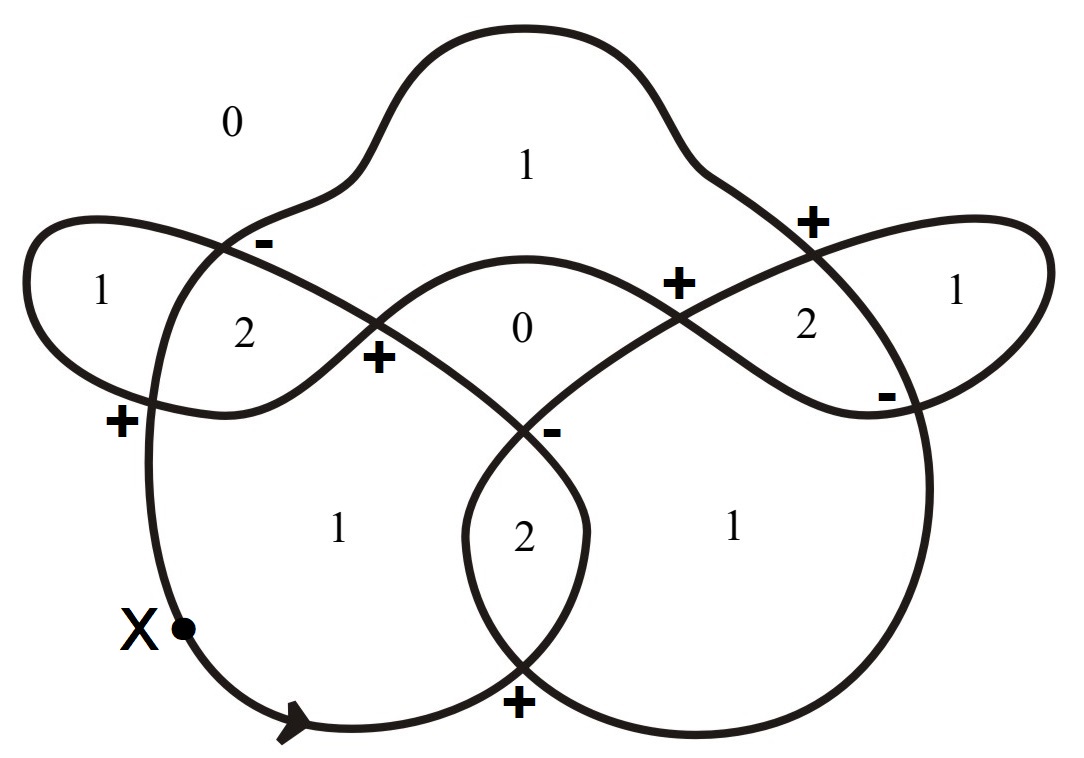}
\end{figure}

\end{ex}

\printbibliography

@article{whitney1937regular,
  title={On regular closed curves in the plane},
  author={Whitney, Hassler},
  journal={Compositio Mathematica},
  volume={4},
  pages={276--284},
  year={1937}
}

@inproceedings{reidemeister1927elementare,
  title={Elementare begr{\"u}ndung der knotentheorie},
  author={Reidemeister, Kurt},
  booktitle={Abhandlungen aus dem Mathematischen Seminar der Universit{\"a}t Hamburg},
  volume={5},
  number={1},
  pages={24--32},
  year={1927},
  organization={Springer}
}

@article{polyak2009minimal,
  title={Minimal sets of Reidemeister moves},
  author={Polyak, Michael},
  journal={arXiv preprint arXiv:0908.3127},
  year={2009},
  publisher={Citeseer}
}

@article{alexander1928topological,
  title={Topological invariants of knots and links},
  author={Alexander, James W},
  journal={Transactions of the American Mathematical Society},
  volume={30},
  number={2},
  pages={275--306},
  year={1928},
  publisher={JSTOR}
}

$$\mbox{Universidad Nacional de Mar del Plata}$$
$$\mbox{CEMIM}$$
$$\mbox{damianwesen@hotmail.com}$$

\end{document}